\documentclass{article}

\usepackage{geometry}
\usepackage{amsmath}
\usepackage{amssymb}
\usepackage{amsthm}
\usepackage{amsfonts}
\usepackage{hyperref}
\usepackage{enumitem}
\usepackage{xcolor}
\usepackage{mathtools}

\renewcommand{\phi}{\varphi}
\renewcommand{\epsilon}{\varepsilon}

\newtheorem{thmCountBase}{}[section]

\newtheorem{lemma}[thmCountBase]{Lemma}
\newtheorem{proposition}[thmCountBase]{Proposition}

\newtheorem{theorem}[thmCountBase]{Theorem}
\theoremstyle{definition}
\newtheorem{definition}[thmCountBase]{Definition}

\newcommand{\bbbrace}{\left\{\begin{aligned}}
\newcommand{\ebbrace}{\end{aligned}\right.}

\newcommand{\bbrace}[1]{\left\{\begin{aligned}#1\end{aligned}\right.}

\newcommand{\afix}[2]{\operatorname{AF}_{#2}(#1)}
\newcommand{\afixb}[2]{{\operatorname{AF}}_{#2}(#1)}
\newcommand{\abs}[1]{\left\lvert#1\right\rvert}

\newcommand{\ballc}[2]{\operatorname{B}_{#1}(#2)}
\newcommand{\ceil}[1]{\left\lceil#1\right\rceil}

\newcommand{\fix}[1]{\operatorname{Fix}(#1)}

\newcommand{\inner}[2]{\left\langle#1,#2\right\rangle}

\newcommand{\N}{\mathbb{N}}

\newcommand{\ol}[1]{\overline{#1}}

\newcommand{\ql}[4]{\inner{\overrightarrow{#1#2}}{\overrightarrow{#3#4}}}
\newcommand{\R}{\mathbb{R}}

\newcommand{\set}[1]{\left\{#1\right\}}
\newcommand{\ul}[1]{\underline{#1}}
\newcommand{\wh}[1]{\widehat{#1}}
\newcommand{\wt}[1]{\widetilde{#1}}

\title{Quantitative metastability of the Tikhonov-Mann iteration for countable families of mappings}
\author{Hora\c{t}iu Cheval$^a$}
\date{$^a$ \footnotesize{Research Center for Logic, Optimization and Security (LOS), Department of Computer Science, \\
Faculty of Mathematics and Computer Science, University of Bucharest \\
Academiei 14, 010014 Bucharest, Romania} \\ \mbox{} \\
E-mail: horatiu.cheval@unibuc.ro
}

\setlist[enumerate, 1]{label={(\roman*)}}

\begin{document}

    \maketitle{}

    \begin{abstract}
        In this paper, we obtain rates of metastability for the 
        Tikhonov-Mann iteration for countable families of mappings in CAT(0) spaces. 
        This iteration was recently defined by the author in the setting of $W$-hyperbolic spaces 
        as a generalization of the strongly convergent version of the Krasnoselskii-Mann iteration 
        introduced by Bo\c{t} and Meier for finding common fixed points of families of nonexpansive mappings in Hilbert spaces,
        and as an extension of the Tikhonov-Mann iteration for single mappings, for which Leu\c{s}tean and the author
        obtained rates of asymptotic regularity in $W$-hyperbolic spaces.
        

    \end{abstract}

    \section{Introduction}

    Bo\c{t} and Meier \cite{BotMei21} considered the following iterative method for 
    finding a common fixed point of a family $(T_n : H \to H)$ of nonexpansive self-mappings of $H$:  
    \begin{align}
        x_0 \in H, \quad x_{n + 1} = (1 - \lambda_n) \beta_n x_n + \lambda_n T_n (\beta_n x_n), \label{eq:tmf-hilbert}
    \end{align}
    where $(\lambda_n)$ and $(\beta_n)$ are sequences in $[0, 1]$. 
    It is proved in \cite[Theorem~1]{BotMei21} that, 
    under certain conditions on the parameters of the iteration and on the family $(T_n)$,
    $(x_n)$ converges strongly to a point in $\bigcap\limits_{n \in \N} \operatorname{Fix}(T_n)$.

    We say that a $W$-hyperbolic space \cite{Koh05} is a metric space $(X, d)$ together with 
    a mapping $W : X \times X \times [0, 1] \to X$, satisfying the following axioms, for all $x, y, z \in X$ and all $\lambda, \theta \in [0, 1]$: \\[2mm]
    \begin{tabular}{lll}
        (W1) & $d(z, (1 - \lambda) x + \lambda y) \leq (1 - \lambda) d(z, x) + \lambda d(z, y)$; \label{w1} \\[2mm] 
        (W2) & $d((1 - \lambda) x + \lambda y, (1 - \theta) x + \theta y) = \abs{\lambda - \theta}d(x, y)$; \label{w2} \\[2mm]
        (W3) & $(1 - \lambda) x + \lambda y = \lambda y + (1 - \lambda) x$; \label{w3}\\[2mm] 
        (W4) & $d((1 - \lambda) x + \lambda z, (1 - \lambda) y + \lambda w) \leq (1 - \lambda) d(x, y) + \lambda d(z, w)$. \label{w4}\\[2mm]
    \end{tabular}
    
    $W(x, y, \lambda)$ is meant to be read as an abstract convex combination between the points $x$ and $y$ of parameter $\lambda$. 
    For this reason, the notation $(1 - \lambda) x + \lambda y := W(x, y, \lambda)$ is normally used instead. 
    Examples of $W$-hyperbolic spaces include real normed spaces, Busemann spaces \cite{Pap05} or CAT(0) spaces \cite{AleKapPet19,BriHae99}. 
    The following generalization to $W$-hyperbolic spaces of iteration \eqref{eq:tmf-hilbert}
    was studied by the author in \cite{Che23}. Let $(X, d, W)$ be a $W$-hyperbolic space and  
    $(T_n : X \to X)$ be a family of nonexpansive mappings.
    The \emph{Tikhonov-Mann iteration} associated to $(T_n)$, of parameters 
    $(\lambda_n), (\beta_n) \subset [0, 1]$ and anchor point $u \in X$, generates a sequence $(x_n)$ by
    \begin{align}
        x_0 \in X, \quad\quad\quad
        \bbrace{
            x_{n + 1} &= (1 - \lambda_n) u_n + \lambda_n T_n u_n, \\ 
            u_n &= (1 - \beta_n) u + \beta_n x_n.
        }
        \label{eq:tmf}
    \end{align}
    Note that if $X$ is a Hilbert space and $u = 0$, one recovers iteration \eqref{eq:tmf-hilbert}.

    The main results of this paper, which are part of the program of \emph{proof mining} \cite{Koh08-book,Koh17, Koh19} concern the metastability (in the sense of Tao \cite{Tao07, Tao08}) 
    of iteration \eqref{eq:tmf} in CAT(0) spaces, with quantitative information in the form of \emph{rates of metastability}. 
    A sequence $(a_n)$ in a metric space $(X, d)$ is said to be \emph{metastable} if  
    \begin{align*}
        \forall k \in \N \forall f \in \N^\N \exists n \in \N \forall i, j \in [n, f(n)] \left(d(a_i, a_j) \leq \frac{1}{k + 1}\right),
    \end{align*}
    where $[n, f(n)]$ is the set $\set{n, n + 1, \hdots, f(n)}$, and a rate thereof is a function providing an
    upper bound on $n$ for each $k$ and $f$.
    Metastability is equivalent to the Cauchy property of $(a_n)$, hence also to its convergence in complete spaces,
    but this equivalence is not effective, meaning that generally one cannot transform rates of metastability 
    into quantitative information on the convergence of $(a_n)$. The importance of metastability 
    in the context of proof mining lies in that, while rates of convergence cannot be extracted in general, 
    by moving to this equivalent formulation, one can guarantee the extractability of rates via logical metatheorems.

    The main results of \cite{Che23} consist of rates of ($(T_n)$-)asymptotic regularity 
    (i.e. rates of convergence for $d(x_n, x_{n + 1}) \to 0$ and $d(x_n, T_n x_n) \to 0$) for \eqref{eq:tmf}.
    The single mapping case $T_n = T$ was studied by Leu\c{s}tean and the author in \cite{CheLeu22}, 
    as a nonlinear generalization of the Krasnoselskii-Mann iteration with Tikhonov regularization 
    terms proved strongly convergent by Bo\c{t}, Csetnek and Meier \cite{BotCseMei19} in Hilbert spaces,
    and for which Dinis and Pinto \cite{DinPin21} computed rates of metastability.
    As shown in \cite{CheKohLeu23}, the strong convergence of the Tikhonov-Mann iteration 
    already follows in CAT(0) spaces from results in \cite{CunPan11}, from which Schade and Kohlenbach 
    extracted rates of metastability in \cite{SchKoh12}.
    Another generalization of the Tikhonov-Mann iteration is the alternating Halpern-Mann iteration 
    introduced by Dinis and Pinto in \cite{DinPin23}, where the authors compute rates of metastability in CAT(0) spaces,
    with rates of asymptotic regularity having been obtained in the more general setting of $UCW$-hyperbolic spaces 
    by Leu\c{s}tean and Pinto in \cite{LeuPin23}.

    \section{Preliminaries}

    As mentioned in the introduction, one important example of $W$-hyperbolic spaces is given by CAT(0) spaces,
    which can be described  \cite[p. 386--388]{Koh08-book} as the $W$-hyperbolic spaces satisfying the following so-called $\operatorname{CN}^-$ inequality,
    for all $x, y, z \in X$:
    \begin{align}
        d^2\left(z, \frac{1}{2} x + \frac{1}{2} y\right) \leq \frac{1}{2} d^2(z, x) + \frac{1}{2} d^2(z, y) - \frac{1}{4} d^2(x, y). \label{eq:cn-minus}
    \end{align}
    Furthermore, \eqref{eq:cn-minus} is in fact enough to derive the following, for all $x, y, z \in X$ and $\lambda \in [0, 1]$.
    \begin{align}
        d^2(z, (1 - \lambda) x + \lambda y) \leq (1 - \lambda) d^2(z, x) + \lambda d^2(z, y) - \lambda (1 - \lambda) d^2(x, y). \label{eq:cn-plus}
    \end{align}

    Leu\c{s}tean \cite{Leu07} showed that CAT(0) spaces are uniformly convex with modulus $\frac{\epsilon^2}{8}$, meaning that, 
    for all $r > 0$, $\epsilon \in (0, 2]$ and all $a, x, y \in X$, 
    \begin{align*}
        \bbrace{
            d(a, x) &\leq r \\ 
            d(a, y) &\leq r \\ 
            d(x, y) &\geq \epsilon r 
        }
        \quad\to\quad d\left(a, \frac{1}{2} x + \frac{1}{2} y\right) \leq \left(1 - \frac{\epsilon^2}{8}\right) r.
    \end{align*}

        
    The following metric generalization of the inner product was introduced in \cite{BerNik08},
    and allows one to extend several arguments from Hilbert space to a nonlinear setting.
    \begin{definition}\label{dfn:quasilinearization}
        Let $(X, d)$ be a metric space. The \emph{Berg-Nikolaev quasilinearization} 
        is the mapping $\ql{\cdot}{\cdot}{\cdot}{\cdot} : X^2 \times X^2 \to \R$ defined by 
        \begin{align}
            \ql{x}{y}{u}{v} = \frac{1}{2}\left(d^2(x, v) + d^2(y, u) - d^2(x, u) - d^2(y, v) \right).
        \end{align}
    \end{definition}    
    The quasilinearization mapping is characterized by the following proposition.
    \begin{proposition}\cite{BerNik08}
        The Berg-Nikolaev quasilinearization is the unique mapping satisfying the following,
        for all $x, y, u, v, w \in X$:
        \begin{enumerate}
            \item $\ql{x}{y}{x}{y} = d^2(x, y)$;
            \item $\ql{x}{y}{u}{v} = \ql{u}{v}{x}{y}$;
            \item $\ql{x}{y}{u}{v} = -\ql{y}{x}{u}{v}$;
            \item $\ql{x}{y}{u}{v} + \ql{x}{y}{v}{w} = \ql{x}{y}{u}{w}$.
        \end{enumerate}
    \end{proposition}
    Furthermore, if $X$ is a CAT(0) space, one has that the following Cauchy-Schwarz inequality holds,
    for all $x, y, u, v \in X$:
    \begin{align*}
        \ql{x}{y}{u}{v} \leq d(x, y) d(u, v).
    \end{align*}

    We now introduce some quantitative notions required to express our results.
    Let $(X, d)$ be a metric space and $(a_n)$ be a sequence in $X$. 
    We say that a function $\phi : \N \to \N$ is a \emph{rate of convergence} for $(a_n)$ to a point $a \in X$ if 
    \begin{align*}
        \forall k \in \N \forall n \geq \phi(k) \left(d(a_n, a) \leq \frac{1}{k + 1}\right),
    \end{align*}
    and that it is a \emph{Cauchy modulus} for $(a_n)$ if 
    \begin{align*}
        \forall k \in \N \forall n \geq \phi(k) \forall j \in \N \left(d(a_{n + j}, a_n) \leq \frac{1}{k + 1}\right).
    \end{align*}
    $(a_n)$ is said to be \emph{asymptotically regular} if $\lim\limits_{n \to \infty} d(a_n, a_{n  +1}) = 0$,
    $T$-asymptotically regular if $\lim\limits_{n \to \infty} d(a_n, T a_n) = 0$
    and $(T_n)$-asymptotically regular if $\lim\limits_{n \to \infty} d(a_n, T_n a_n) = 0$,
    where $T : X \to X$ and $(T_n : X \to X)_{n \in \N}$. A rate of ($T$-, $(T_n)$-)asymptotic regularity 
    is a rate of convergence to $0$ for the respective sequences.

    A function $\mu : \N \times \N^\N \to \N$ is called a \emph{rate of metastability} for the sequence $(a_n)$ if 
    \begin{align*}
        \forall k \in \N \forall f \in \N^\N \exists n \leq \mu(k, f) \forall i, j \in [n, f(n)] \left(d(a_i, a_j) \leq \frac{1}{k + 1}\right),
    \end{align*}
    where $[n, f(n)] = \set{n, n + 1, \hdots, f(n)}$.
    We may restrict the definition of metastability to only quantify over monotone (i.e. nondecreasing) functions $f : \N \to \N$,
    as we can replace $f$ with the monotone function $f^M(k) = \max\limits_{i \leq k} f(i)$.

    Finally, a \emph{rate of divergence} for a series $\sum\limits_{n = 0}^\infty b_n$ of nonnegative real numbers 
    is a function $\theta : \N \to \N$ such that 
    \begin{align*}
        \forall n \in \N \left(\sum_{i = 0}^{\theta(n)} b_i \geq n \right).
    \end{align*}

    The main convergence theorem of \cite{BotMei21} relies 
    on a widely used lemma on real numbers by Xu \cite{Xu02}.
    Several quantitative versions of \cite[Lemma~2.5]{Xu02} 
    have been given \cite{KohLeu12, LeuPin21} in the context of proof mining.
    The particular variant we will use is the special case with $\gamma_n = 0$ of \cite[Lemmas~14,~16]{Pin21} 
    (see also \cite[Lemma~2.6]{DinPin21}).

    \begin{lemma}\label{lem:xu-metastable}
        Consider the sequences $(s_n) \subset [0, \infty)$, $(a_n) \subset (0, 1)$, 
        $(r_n)$ and $(v_n) \subset \R$ 
        such that $(s_n)$ is bounded above by $S \in \N$ and 
        that, for all $n \in \N$, the following inequality holds
        \begin{align}
            s_{n + 1} \leq (1 - a_n) (s_n + v_n) + a_n r_n.
        \end{align}
        Let $k, n, q \in \N$ be such that 
        \begin{enumerate}
            \item \label{lem:xu-metastable-2}
            $\forall i \in [n, q] \left(v_i \leq \dfrac{1}{3(k + 1)(q + 1)}\right)$;
            \item \label{lem:xu-metastable-3} 
            $\forall i \in [n, q] \left(r_i \leq \dfrac{1}{3(k + 1)}\right)$.
        \end{enumerate}
        Then, following hold:
        \begin{enumerate}
            \item Suppose $\sum\limits_{n = 0}^\infty a_n$ diverges with a monotone rate $\sigma$. Then  
            \begin{align}
                \forall i \in [\zeta(k, n), q] \left(s_i \leq \frac{1}{k + 1}\right),
            \end{align}
            where $\zeta(k, n) := \sigma(n + \ln\ceil{(3S(k + 1))}) + 1$.

            \item
             Suppose $\sigma^* : \N \times \N \to \N$ is a monotone function such that, for any $m \in \N$, 
             $\sigma^*(m, \cdot)$ is a rate of convergence for $\prod\limits_{n = m}^\infty (1 - a_n) = 0$.
            Then, 
            \begin{align}
                \forall i \in [\zeta(k, n), q] \left(s_i \leq \frac{1}{k + 1}\right),
            \end{align}
            where $\zeta^*(k, n) := \sigma^*(n, 3S(k + 1) - 1) + 1$.
        \end{enumerate}
    \end{lemma}

    The following conditions on the parameters of iteration \eqref{eq:tmf},
    which are quantitative counterparts of those used in \cite[Theorem~1]{BotMei21},
    will be considered throughout.
    
    \newcommand{\qcondSumOneMinusBetaDivergent}{(C1_q)}
    \newcommand{\qcondProdBetaZero}{(C1_q^*)}
    \newcommand{\qcondSumConseqLambdaConvergent}{(C3_q)}
    \newcommand{\qcondSumConseqBetaConvergent}{(C2_q)}
    \newcommand{\qcondLimBetaOne}{(C4_q)}
    \newcommand{\qcondLiminfBetaGtZero}{(C5_q)} 
    \newcommand{\qcondBetaGtZero}{(C9_q)}

    \newcommand{\rSumOneMinusBetaDivergent}{\sigma}
    \newcommand{\rProdBetaZero}{\sigma^*}
    \newcommand{\rSumConseqLambdaConvergent}{\chi_\lambda}
    \newcommand{\rSumConseqBetaConvergent}{\chi_\beta}
    \newcommand{\rLimBetaOne}{\eta}
    \newcommand{\rGammaConseqCauchy}{\chi_\gamma}
    
    \newcommand{\qcondTnConseqCauchy}{(C6_q)}
    \newcommand{\rTnConseqCauchy}{\chi_T}
    
    \begin{tabular}{lll}
        $\qcondSumOneMinusBetaDivergent$ & $\sum\limits_{n = 0}^\infty (1 - \beta_n)$ diverges with monotone rate of divergence $\rSumOneMinusBetaDivergent$; \\[3mm] 
        $\qcondProdBetaZero$ & $\rProdBetaZero : \N \times \N \to \N$ is a monotone function such that \\ & for any $m \in \N$, $\prod\limits_{n = m}^\infty \beta_{n} = 0$ with rate of convergence $\rProdBetaZero(m, \cdot)$; \\[3mm] 
        $\qcondSumConseqBetaConvergent$ & $\sum\limits_{n = 0}^\infty \abs{\beta_n - \beta_{n + 1}}$ converges with Cauchy modulus $\rSumConseqBetaConvergent$; \\[3mm]
        $\qcondSumConseqLambdaConvergent$ & $\sum\limits_{n = 0}^\infty \abs{\lambda_n - \lambda_{n + 1}}$ converges with Cauchy modulus $\rSumConseqLambdaConvergent$; \\[3mm]
        $\qcondLimBetaOne$ & $\lim\limits_{n \to \infty} \beta_n = 1$ with rate of convergence $\rLimBetaOne$; \\[3mm] 
        $\qcondLiminfBetaGtZero$ & $\Lambda \in \N^*$ and $N_\Lambda \in \N$ are such that $\lambda_n \geq \frac{1}{\Lambda}$ for all $n \geq N_\Lambda$; \\[3mm]
        $\qcondTnConseqCauchy$ & $\sum\limits_{n = 0}^\infty d(T_{n + 1} u_n, T_n u_n)$ converges with Cauchy modulus $\rTnConseqCauchy$. \\[3mm]
    \end{tabular}
    
    The conditions above were used in \cite{Che23} to obtain rates of ($(T_n)$-)asymptotic regularity for $(x_n)$
    in general $W$-hyperbolic spaces, with the caveat that condition $\qcondProdBetaZero$ here 
    is a different quantitative formulation of the fact that $\prod\limits_{n = 0}^\infty \beta_n = 0$ from the one used in \cite{Che23}.
    This is further equivalent to the assumption that $\sum\limits_{n = 0}^\infty (1 - \beta_n) = \infty$,
    but, as first observed by Kohlenbach \cite{Koh11}, their quantitative content may in general be different.
    Let us now recall the following sufficient conditions for the family $(T_n)$ to satisfy $\qcondTnConseqCauchy$.
    \begin{proposition}\cite[Proposition~3.4]{Che23}\label{prop:ggg}
        Let $(\gamma_n)$ be a sequence in $(0, \infty)$ satisfying:

        \begin{tabular}{lll}
            $(C7_q)$ $\sum\limits_{n = 0}^\infty \abs{\gamma_n - \gamma_{n + 1}}$ is convergent with Cauchy modulus $\rGammaConseqCauchy$;\\[3mm]
            $(C8_q)$ $\Gamma \in \N^*$ and $N_\Gamma \in \N$ are such that $\gamma_n \geq \frac{1}{\Gamma}$ for all $n \geq N_\Gamma$. \\[3mm]
        \end{tabular}

        Let $p \in \bigcap_{n \in \N} \operatorname{Fix(T_n)}$ be a common fixed point and suppose 
        that $(T_n)$ satisfies the following,
        for all $n, m \in \N$ and $x \in X$: 
        \begin{align}
            d(T_n x, T_m x) \leq \frac{\abs{\gamma_m - \gamma_n}}{\gamma_n} d(T_n x, x). \label{eq:prePtwo}
        \end{align}
        Then, $\rTnConseqCauchy$ defined below is a Cauchy modulus for $\sum\limits_{n = 0}^\infty d(T_{n + 1} u_n, T_n u_n)$:
        \begin{align}
            \rTnConseqCauchy(k) = \max\set{N_\Gamma, \rGammaConseqCauchy(2K\Gamma(k + 1) - 1)},
        \end{align}
        where $K \geq \max\set{d(x_0, p), d(u, p)}$.
    \end{proposition}

    Condition \eqref{eq:prePtwo} above was introduced by Leu\c{s}tean, Nicolae and Sipo\c{s} in \cite{LeuNicSip18},
    called Condition $(C1)$ there. 
    It is shown in \cite{LeuNicSip18} that if $(T_n)$ is jointly $(P_2)$ (in particular jointly firmly nonexpansive) with respect to a sequence $(\gamma_n) \subset (0, \infty)$,
    then it satisfies \eqref{eq:prePtwo}. 
    Concrete examples of families that satisfy \eqref{eq:prePtwo} given in \cite{LeuNicSip18}
    include: proximal mappings of proper, convex, lower semicontinuous functions, 
    resolvent of nonexpansive self-mappings of CAT(0) spaces,
    and resolvents of maximally monotone operators on Hilbert spaces.
    
    Additionally, we will also use the following condition on $(\beta_n)$ \\[2mm]
    \begin{tabular}{lll}
        $\qcondBetaGtZero$ & $B : \N \to \N^*$ is such that $\forall n \in \N \left(\beta_n \geq \frac{1}{B(n)}\right)$. \\[3mm]
    \end{tabular}

\section{Main results}

From here on, suppose the family $(T_n)$ has common fixed points and let $p \in X$ be such a point. 
Let $M$ be defined by 
\begin{align}
    M = \max\set{d(x_0, p), d(u, p)}.
\end{align} 
and $K \in \N$ be such that $K \geq M$.

The main results of \cite{Che23} provide rates of asymptotic regularity 
and $(T_n)$-asymptotic regularity for $(x_n)$, 
under conditions $\qcondProdBetaZero$ -- $\qcondTnConseqCauchy$, 
with a different formulation of $\qcondProdBetaZero$, as explained. 
It is straightforward to restate the asymptotic regularity results from \cite[Theorems~3.5, 3.7]{Che23} 
with the reformulated condition $\qcondProdBetaZero$ used here, and to also give an analogous result for condition $\qcondSumOneMinusBetaDivergent$, as follows.


\begin{theorem}
    Suppose conditions $\qcondSumConseqBetaConvergent$, $\qcondSumConseqLambdaConvergent$, and $\qcondTnConseqCauchy$ hold,
    and define 
    \begin{align*}
        \chi(k) = \max\set{\chi_T(2(k + 1) - 1), \chi_\lambda(8K(k + 1) - 1), \chi_\beta(8K(k + 1) - 1)}.
    \end{align*}
    The following hold: 
    \begin{enumerate}
    \item If $\qcondSumOneMinusBetaDivergent$ is satisfied, then 
        $(x_n)$ is asymptotically regular with rate 
        \begin{align*}
            \Sigma(k) = \rSumOneMinusBetaDivergent(\chi(3k + 2) + 2+ \ceil{\ln(6K(k + 1))}) + 1.
        \end{align*}

        If, furthermore, conditions $\qcondLimBetaOne$ and $\qcondLiminfBetaGtZero$ hold, then 
        $(x_n)$ is $(T_n)$-asymptotically regular with rate        
        \begin{align*}
            \wt{\Sigma}(k) = \max\set{N_\Lambda, \Sigma(2\Lambda(k + 1) - 1), \rLimBetaOne(4K\Lambda(k + 1) - 1)}.
        \end{align*}

    \item Suppose $\qcondProdBetaZero$ holds instead of $\qcondSumOneMinusBetaDivergent$. Then, $(x_n)$ is asymptotically regular with rate 
        \begin{align*}
            \Sigma^*(k) = \rProdBetaZero(\chi(3k + 2), 6K(k + 1) - 1) + 1,
        \end{align*}
        If, furthermore, conditions $\qcondLimBetaOne$ and $\qcondLiminfBetaGtZero$ hold, then 
        $(x_n)$ is $(T_n)$-asymptotically regular with rate        
        \begin{align*}
            \wt{\Sigma^*}(k) = \max\set{N_\Lambda, \Sigma^*(2\Lambda(k + 1) - 1), \rLimBetaOne(4K\Lambda(k + 1) - 1)}.
        \end{align*}

    \end{enumerate}
\end{theorem}
\begin{proof}
    For the asymptotic regularity results, 
    the proof is the same as that of \cite[Theorem~3.5]{Che23},
    using Lemmas 13 and 14 from \cite{Pin21}, respectively. 
    The $(T_n)$-asymptotic regularity then follows identically to \cite[Theorem~3.7]{Che23}.
\end{proof}


The following shows that, under certain conditions, one can also obtain  
rates of $T_m$-asymptotic regularity, for any mapping $T_m$ in the family.

\begin{theorem}\label{thm:Tn-ar-ar}
    Let $(\gamma_n) \subset (0, \infty)$ be a sequence satisfying $(C8_q)$, 
    furthermore bounded above by some $G \in \N$.
    Suppose that $(T_n)$ satisfies \eqref{eq:prePtwo} with respect to $(\gamma_n)$
    and let $\phi$ be a rate of $(T_n)$-asymptotic regularity for $(x_n)$.
    Then, for any $m \in \N$,
    $(x_n)$ is $T_m$-asymptotically regular with rate 
    \begin{align*}
        k \mapsto \max\set{\phi((1 + 2\Gamma G) (k + 1) - 1), N_\Gamma}.
    \end{align*}
\end{theorem}
\begin{proof}
    Let $k \geq \max\set{\phi((1 + 2\Gamma G) (k + 1) - 1), N_\Gamma}$. 
    We have that 
    \begin{align*}
        d(x_n, T_m x_n) \leq&\ d(x_n, T_n x_n) + d(T_n x_n, T_m x_n) \\ 
        \leq&\ d(x_n, T_n x_n) + \frac{\abs{\gamma_m - \gamma_n}}{\gamma_n} d(T_n x_n, x_n) \quad\text{by \eqref{eq:prePtwo}} \\ 
        \leq&\ d(x_n, T_n x_n) + \Gamma {\abs{\gamma_m - \gamma_n}} d(T_n x_n, x_n)  \\
        \leq&\ d(x_n, T_n x_n) + \Gamma (\abs{\gamma_m} + \abs{\gamma_n}) d(T_n x_n, x_n) \\ 
        \leq&\ (1 + 2 \Gamma G) d(T_n x_n, x_n) \quad\text{since $\abs{\gamma_n}, \abs{\gamma_m} \leq G$}\\ 
        \leq&\  (1 + 2 \Gamma G) \frac{1}{(1 + 2 \Gamma G) (k + 1)} = \frac{1}{k + 1} 
    \end{align*}
    where the last inequality holds since $n \geq \phi((1 + 2\Gamma G) (k + 1) - 1)$.
\end{proof}

Combining this with the previous theorem, we obtain the following.
\begin{theorem}\label{thm:Tm-ar}
    Let $(X, d, W)$ be a $W$-hyperbolic space, $(T_n : X \to X)$ be a family of nonexpansive mappings 
    and let $(x_n)$ be the sequence generated by \eqref{eq:tmf}. 
    Let $(\gamma_n)$ be a sequence of positive reals such that $(T_n)$ satisfies \eqref{eq:prePtwo} 
    with respect to $(\gamma_n)$, and assume conditions $(C2_q)$ -- $(C6_q)$ and $(C8_q)$ hold and that $G \in \N$ is an upper bound on $(\gamma_n)$
    \begin{enumerate}
        \item If $\qcondSumOneMinusBetaDivergent$ is satisfied, then for any $m \in \N$, $(x_n)$ is $T_m$-asymptotically regular with rate 
        \begin{align}
            \Psi(k) = \max\set{\wt{\Sigma}((1 + 2\Gamma G) (k + 1) - 1), N_\Gamma},
        \end{align}
        where $\wt{\Sigma}$ is defined as in Theorem \ref{thm:Tn-ar-ar}.(i).

        \item Analogously, if $\qcondProdBetaZero$ is satisfied, then for any $m \in \N$, $(x_n)$ is $T_m$-asymptotically regular with rate 
        \begin{align}
            \Psi^*(k) = \max\set{\wt{\Sigma^*}((1 + 2\Gamma G) (k + 1) - 1), N_\Gamma},
        \end{align}
        where $\wt{\Sigma^*}$ is defined as in Theorem \ref{thm:Tn-ar-ar}.(ii).
    \end{enumerate}
\end{theorem}

We can now proceed to prove our main metastability result. 
The proofs are essentially an adaptation of those in \cite{DinPin21,DinPin23}.
A general account in the context of proof mining for this type of techniques, 
first used by Kohlenbach \cite{Koh11}, is given by Ferreira, Leu\c{s}tean and Pinto in \cite{FerLeuPin19}.

For any $k \geq 1$, let us denote by 
\begin{align*}
    \afixb{k}{} = \set{x \in X \mid \forall n \in \N \left(d(x, T_n x) \leq \frac{1}{k}\right)} \cap B_p(K)    
\end{align*}
the set of common $\frac{1}{k}$-approximate fixed points of $(T_n)$ 
which also belong to the closed ball $B_p(K)$ of radius $K$ and center $p$. 
We first prove some recursive inequalities involving the sequence $(x_n)$,
which are essentially a generalization to our setting of those used for \cite[Theorem~3.6]{DinPin21}.

\begin{proposition}
    Let $x \in X$ be any point. 
    Then, for all $n \in \N$,
    \begin{enumerate}
        \item $d(x_{n + 1}, x) \leq d(u_n, x) + d(x, T_n x)$;
        \item $d^2(u_n, x) \leq \beta_n d^2(x_n, x) + 2\beta_n(1 - \beta_n)\ql{x}{u}{x}{x_n} + (1 - \beta_n)^2 d^2(x, u)$;
        \item $d^2(x_{n + 1}, x) \leq \beta_n(d^2(x_n, x) + B(n) w_n) + (1 - \beta_n) (2\beta_n\ql{x}{u}{x}{x_n}) + (1 - \beta_n)d^2(x, u)$,
    \end{enumerate}
    where $w_n := 2d(u_n, x)d(T_n x, x) + d^2(T_n x, x)$.
\end{proposition}
\begin{proof}
    Let $x \in X$ and $n \in \N$. We have
    \begin{enumerate}
        \item
        \begin{align*}
            d(x_{n + 1}, x) \leq&\ (1 - \lambda_n) d(u_n, x) + \lambda_n d(T_n u_n, x) \quad\text{by (W1)} \\ 
            \leq&\ (1 - \lambda_n) d(u_n, x) + \lambda_n d(T_n u_n, T_n x) + \lambda_n d(T_n x, x) \\ 
            \leq&\ (1 - \lambda_n) d(u_n, x) + \lambda_n d(u_n, x) + d(T_n x, x) \quad\text{by nonexpansiveness} \\ 
            =&\ d(u_n, x) + d(T_n x, x) 
        \end{align*}

        \item 
        \begin{align*}
            & d^2(u_n, x) - \beta_n d^2(x_n, x) + 2\beta_n(1 - \beta_n)\ql{x}{u}{x}{x_n} + (1 - \beta_n)^2 d^2(x, u) \\
            =&\ d^2(u_n, x) - \beta_n d^2(x_n, x) - \beta_n(1 - \beta_n) \beta_n(1 - \beta_n)(d^2(x, x_n) + d^2(u, x) - d^2(u, x_n)) - (1 - \beta_n)^2 d^2(x, u) \\ 
            &\quad\text{by the definition of $\ql{\cdot}{\cdot}{\cdot}{\cdot}$} \\ 
            \leq&\ (1 - \beta_n)d^2(u, x) + \beta_n d^2(x_n, x) - \beta_n(1- \beta_n)d^2(u, x_n) \quad\text{by \eqref{eq:cn-plus}} \\ 
            &- \beta_n(1 - \beta_n) \beta_n(1 - \beta_n)(d^2(x, x_n) + d^2(u, x) - d^2(u, x_n)) - (1 - \beta_n)^2 d^2(x, u) \\ 
            =&\ (1 - \beta_n) d^2(u, x) - \beta_n(1 - \beta_n)(d^2(x, x_n) + d^2(u, x)) - (1 - \beta_n)^2 d^2(x, u) \\ 
            =&\ \beta_n (1 - \beta_n) d^2(u, x) - \beta_n(1 - \beta_n)(d^2(x, x_n) + d^2(u, x)) \\ 
            =&\ \beta_n (1 - \beta_n)(d^2(u, x) - d^2(x, x_n) - d^2(u, x)) \\ 
            =&\ \beta_n(1 - \beta_n) (- d^2(x, x_n)) \\ 
            \leq&\ 0
        \end{align*}

        \item 
        \begin{align*}
            d^2(x_{n + 1}, x) \leq&\ (d(u_n, x) + d(T_n x, x))^2 \quad\text{by (i)} \\
            =&\ d^2(u_n, x) + 2d(u_n, x)d(T_n x, x) + d^2(T_n x, x) \\
            =&\ d^2(u_n, x) + w_n \\ 
            \leq&\ \beta_n d^2(x_n, x) + 2\beta_n(1 - \beta_n)\ql{x}{u}{x}{x_n} + (1 - \beta_n)^2 d^2(x, u) + w_n \quad\text{by (ii)} \\ 
            =&\ \beta_n d^2(x_n, x) + (1 - \beta_n) (2\beta_n\ql{x}{u}{x}{x_n} + (1 - \beta_n) d^2(x, u)) + w_n \\ 
            =&\ \beta_n (d^2(x_n, x) + \frac{w_n}{\beta_n}) + (1 - \beta_n) (2\beta_n\ql{x}{u}{x}{x_n} + (1 - \beta_n) d^2(x, u)) \\ 
            \leq&\ \beta_n (d^2(x_n, x) + B(n) {w_n}) + (1 - \beta_n) (2\beta_n\ql{x}{u}{x}{x_n} + (1 - \beta_n) d^2(x, u)),
        \end{align*}
        where the last inequality uses Condition $\qcondBetaGtZero$.
    \end{enumerate}
\end{proof}

The following lemma is analogous to \cite[Proposition~4.1]{DinPin23} (which in turn generalizes \cite[Proposition~3.1]{FerLeuPin19} to CAT(0) spaces)
and the proof is easily adapted. 
We give it here for completeness.
    
\begin{lemma}\label{prop:tkm:metastability:afix-pojection}
    Let $k \in \N$ and $f : \N \to \N$ be a monotone function and set  
    $r(k) := K^2(k + 1)$. 
    There exists $N \leq f^{(r(k))}(0)$ and $x \in \afixb{f(N) + 1}{}$ such that
    \begin{align*}
        \forall y \in \afixb{N + 1}{} \left(d^2(x, u) \leq d^2(y, u) + \frac{1}{k + 1}\right).
    \end{align*}
\end{lemma}
\begin{proof}
    We argue by contradiction. Let therefore $k \in \N$ and $f : \N \to \N$ monotone be such that,
    for all $N \leq f^{(r(k))}(0)$ and all $x \in \afixb{f(N) + 1}{}$,
    there exists $y \in \afixb{N + 1}{}$ such that 
    \begin{align}
        d^2(x, u) > d^2(y, u) + \frac{1}{k + 1}. \label{eq:tkm:metastability:afix-pojection-negated-conclusion}
    \end{align}
    We recursively define a finite sequence $v_0, \hdots, v_{r(k)}, v_{r(k) + 1}$,
    with the properties that
    \begin{align}
        \forall j \in [0, r(k) + 1] \forall n \in \N \left(d(v_j, T_n v_j) \leq \frac{1}{f^{(r(k) - j + 1)}(0) + 1}\right) \label{eq:tkm:metastability:afix-pojection:rec-cond1}
    \end{align}
    and that 
    \begin{align}
        \forall j \in [0, r(k)] \left(d^2(v_{j + 1}, u) \leq d^2(v_j, u) - \frac{1}{k + 1}\right), \label{eq:tkm:metastability:afix-pojection:rec-cond2}
    \end{align}
    as follows: \\ 
    $\ul{v_0}$: \\ 
    Choose $v_0 := p$. Condition \eqref{eq:tkm:metastability:afix-pojection:rec-cond1} is clearly satisfied as $p \in \fix{T_n}$ for any $n \in \N$.

    \noindent$\ul{\text{$v_{j + 1}$, for $j \leq r(k)$}}$: \\ 
    We have some $v_j$ satisfying \eqref{eq:tkm:metastability:afix-pojection:rec-cond1} and \eqref{eq:tkm:metastability:afix-pojection:rec-cond2}.
    Set $v_{j + 1}$ to be the $y$ given by \eqref{eq:tkm:metastability:afix-pojection-negated-conclusion}
    for $N := f^{(r(k) - j)}(0)$ and $x := v_j$. Note that $f^{(r(k) - j)}(0) \leq f^{(r(k))}(0)$ by the monotonicity of $f$.
    Conditions \eqref{eq:tkm:metastability:afix-pojection:rec-cond1} and \eqref{eq:tkm:metastability:afix-pojection:rec-cond2}
    are then satisfied due to \eqref{eq:tkm:metastability:afix-pojection-negated-conclusion}.
    Using property \eqref{eq:tkm:metastability:afix-pojection:rec-cond2} of the sequence $(v_j)$, 
    we obtain the following contradiction: 
    \begin{align*}
        d^2(v_{r(k) + 1}, u) 
        &\leq\ d^2(v_0, u) - \frac{r(k) + 1}{k + 1} \\ 
        &=\ d^2(p, u) - \frac{r(k) + 1}{k + 1} \quad\text{since $K \geq M$}\\
        &\leq\ K^2 - \frac{K^2(k + 1) + 1}{k + 1} = \frac{-1}{k + 1} < 0.
    \end{align*}
\end{proof}

The following is an adaptation to our case of \cite[Lemma~4.2]{DinPin23}, which in turn generalizes from Hilbert to CAT(0) spaces 
Lemmas 2.3 and 2.7 from \cite{Koh11}. 
The proof for our countable mappings case remains essentially the same, and we omit it here.

\begin{lemma}\label{prop:tmf:metastability:afix-convex}
    For any $k \in \N$ and all $v_1, v_2 \in \ballc{p}{K}$, we have that
    \begin{align*}
        \forall n \in \N \left(\bigwedge\limits_{i = 1}^2 d(v_i, T_n v_i) < \frac{1}{\omega_1(k)}\right)
            \to 
        \forall n \in \N \forall t \in [0, 1] \left(d(w_t, T_n w_t) < \frac{1}{k + 1}\right),
    \end{align*}
    where $\omega_1(k) = 24K(k + 1)^2$ and $w_t = (1 - t) v_1 + t v_2$.
\end{lemma}

The next result follows as in \cite[Lemma~4.3]{DinPin23}.
\begin{lemma}\label{prop:tmf:metastability:proj-to-proj-variational}
    For all $k \in \N$, $x, y \in \ballc{p}{K}$ and $u \in X$,
    \begin{align*}
        \forall t \in [0, 1] \left(d^2(x, u) \leq d^2(w_t, u) + \frac{1}{\omega_2(k)}\right) \to \ql{x}{u}{x}{y} \leq \frac{1}{k + 1},
    \end{align*}
    where $\omega_2(k) = 4K^2(k + 1)^2$ and $w_t = (1 - t) x + t y$.
\end{lemma}

In the following, we will use the following notation: 
given a function $f : \N \to \N$, we write $\wh{f}(k) = \max\set{\omega_1(k), f(\omega_1(k))}$,
where $\omega_1$ is defined as in Lemma \ref{prop:tmf:metastability:afix-convex}.
\begin{lemma}\label{prop:tmf:metastability:proj-on-interval}
    For all $k \in \N$ and $f : \N \to \N$ monotone, 
    there exists $N \leq \omega_1(\wh{f}^{(r(k))}(0))$ and $x \in \afixb{f(N) + 1}{}$ such that,
    for all $y \in \afixb{N + 1}{}$, we have that 
    \begin{align}
        \forall t \in [0, 1] \left(d^2(x, u) < d^2(w_t, u) + \frac{1}{k + 1}\right). \label{eq:prop:tmf:metastability:proj-on-interval:local-needed}
    \end{align}
\end{lemma}
\begin{proof}
    Let $k \in \N$ and $f : \N \to \N$ be a monotone function.
    Applying Lemma \ref{prop:tkm:metastability:afix-pojection} with 
    $k := k$ and $f := \wh{f}$ yields an 
    $N_0 \leq \wh{f}^{(r(k))}(0)$ and an 
    $x \in \afixb{\wh{f}(N_0) + 1}{}$ such that, 
    \begin{align}
        \forall y \in \afixb{N_0 + 1}{} \left(d^2(x, u) \leq d^2(y, u) + \frac{1}{k + 1}\right) \label{eq:prop:tmf:metastability:proj-on-interval:local-hypothesis}
    \end{align}
    We choose $N := \omega_1(N_0)$ and $x := x$ and we prove that they satisfy the claim.
    Indeed, since $\omega_1$ is monotone and $N_0 \leq \wh{f}^{(r(k))}(0)$, 
    it follows that $N = \omega_1(N_0) \leq \omega_1(\wh{f}^{(r(k))}(0))$.
    Note also that $f(N) = f(\omega_1(N_0)) \leq \max\set{\omega_1(N_0), f(\omega_1(N_0))} = \wh{f}(N_0)$.
    Therefore, for all $n \in \N$,
    \begin{align}
        d(x, T_n x) \leq \frac{1}{\wh{f}(N_0) + 1} \leq \frac{1}{f(N) + 1}
    \end{align}
    showing that $x$ is indeed a $\dfrac{1}{f(N) + 1}$-approximate common fixed point of $(T_n)$.
    We are left to show that \eqref{eq:prop:tmf:metastability:proj-on-interval:local-needed} is also satisfied. 
    Let $y \in \afixb{N + 1}{}$.
    We have, for all $n \in \N$, that 
    \begin{align*}
        d(x, T_n x) &\leq \frac{1}{\wh{f}(N_0) + 1} = \frac{1}{\max\set{\omega_1(N_0), f(\omega_1(N_0))} + 1} \leq \frac{1}{\omega_1(N_0)}
    \intertext{and that}
        d(y, T_n y) &\leq \frac{1}{N + 1} = \frac{1}{\omega_1(N_0) + 1} \leq \frac{1}{\omega_1(N_0)}.
    \end{align*}
    We have thus shown that $x$ and $y$ are both $\dfrac{1}{\omega_1(N_0)}$-approximate common fixed points.
    We can therefore apply Lemma \ref{prop:tmf:metastability:afix-convex} with 
    $k := N_0$, $v_1 := x$ and $v_2 := y$ to get that, for all $n \in \N$, 
    $d(w_t, T_n w_t) \leq \dfrac{1}{N_0 + 1}$. 
    Thus, \eqref{eq:prop:tmf:metastability:proj-on-interval:local-hypothesis} may be applied to $y := w_t$ in order to get the desired conclusion.
\end{proof}

The following is due to \cite[Proposition~4.4]{DinPin23} and combines the previous two results.

\begin{lemma}\label{prop:tmf:metastability:afix-proj-variational}
    For any $k \in \N$ and monotone function $f : \N \to \N$,
    there exists $N \leq \omega_1(\wh{f}^{(r(\omega_2(k)))}(0))$ and 
    $x \in \afix{f(N) + 1}{}$ such that 
    \begin{align}
        \forall y \in \afixb{N + 1}{} \left(\ql{x}{u}{x}{y} \leq \frac{1}{k + 1}\right) 
    \end{align}
\end{lemma}
\begin{proof}
    Let $k \in \N$ and $f : \N \to \N$ be a monotone function.
    Apply Lemma \ref{prop:tmf:metastability:proj-on-interval} with $k := \omega_2(k)$ and $f := f$ 
    to get an $N \leq \omega_1(\wh{f}^{(r(\omega_2(k)))}(0))$ 
    and an $x \in \afixb{f(N) + 1}{}$ such that 
    \begin{align}
        \forall y \in \afixb{N + 1}{} \forall t \in [0, 1] \left(d^2(x, u) < d^2(w_t, u) + \frac{1}{\omega_2(k) + 1}\right), \label{eq:tmf:metastability:afix-proj-variational:local}
    \end{align}
    where $w_t = (1 - t) x + t y$. 
    Apply Lemma \ref{prop:tmf:metastability:proj-to-proj-variational} to \eqref{eq:tmf:metastability:afix-proj-variational:local}
    to get that 
    \begin{align*}
        \ql{x}{u}{x}{y} \leq \frac{1}{k + 1} 
    \end{align*}
    concluding the proof.
\end{proof}

The final lemma needed for the main result is the following, 
which generalizes to the setting of countable families \cite[Proposition~4.5]{DinPin23},
making use of a rate of $T_m$-asymptotic regularity for $(x_n)$.

\begin{lemma}\label{prop:tmf:metastability:weak-seq-conv-elim}
    For all $k \in \N$ and monotone function $f : \N \to \N$, 
    there exists $N \leq \omega_3(k, f)$ and $x \in \afixb{f(N) + 1}{}$ such that 
    \begin{align}
        \forall i \geq N \left(\ql{x}{u}{x}{x_i} \leq \frac{1}{k + 1}\right),
    \end{align}
    where $\omega_3(k, f) = \Phi(\omega_1(\wh{f \circ \Phi}^{(r(\omega_2(k)))}(0)))$,
    and $\Phi$ is, for any $m \in \N$, a monotone rate of $T_m$-asymptotic regularity for $(x_n)$. 
\end{lemma}
\begin{proof}
    Let $k \in \N$ and $f : \N \to \N$ be a monotone function.
    Apply Proposition \ref{prop:tmf:metastability:afix-proj-variational} with 
    $k := k$ and $f := f \circ \Phi$ to get an 
    $N_0 \leq \omega_1(\wh{f \circ \Phi}^{(r(\omega_2(k)))}(0))$ 
    and an $x \in \afixb{f(\Phi(N_0)) + 1}{}$ such that 
    \begin{align}
        \forall y \in \afixb{N_0 + 1}{} \left(\ql{x}{u}{x}{y} \leq \frac{1}{k + 1}\right). \label{eq:tmf:metastability:weak-seq-conv-elim:local}
    \end{align}
    Set $N := \Phi(N_0)$. Since $N_0 \leq \omega_1(\wh{f \circ \Phi}^{(r(\omega_2(k)))}(0))$,
    the monotonicity of $\Phi$ implies that $N \leq \omega_3(k, f)$.
    Since for any $m \in \N$, 
    $\Phi$ is a rate of $T_m$-asymptotic regularity for $(x_n)$, we know that  for any $m \in \N$,
    \begin{align*}
        \forall i \geq N \left(d(x_i, T_m x_i) \leq \frac{1}{N_0 + 1}\right).
    \end{align*}
    In other words, for all $i \geq N$, $x_i$ is a $\dfrac{1}{N_0 + 1}$-approximate common fixed point.
    We can therefore apply \eqref{eq:tmf:metastability:weak-seq-conv-elim:local} with $y := x_i$, for $i \geq N$,
    to get that 
    \begin{align*}
        \forall i \geq N \left(\ql{x}{u}{x}{x_i} \leq \frac{1}{k + 1}\right),
    \end{align*}
    concluding the proof.
\end{proof}
Note that the monotonicity requirement in the previous lemma can always be circumvented 
since if $\Phi$ is a rate of $T_m$-asymptotic regularity, then so is the monotone function
$\Phi^M(k) = \max\limits_{i \leq k} \Phi(i)$.
We are now in the position to obtain rates of metastability for $(x_n)$ in the next theorem.

\begin{theorem}
    Let $X$ be a CAT(0) space, $(T_n : X \to X)$ be a family of nonexpansive self-mappings of $X$
    possessing common fixed points and let $p \in X$ be such a point.
    Suppose that $(T_n)$
    satisfies \eqref{eq:prePtwo} with respect to a sequence $(\gamma_n)$ of positive reals.
    Let $(x_n)$ be the sequence generated by \eqref{eq:tmf} and
    assume conditions $(C1_q)$, $(C2_q)$ -- $(C6_q)$, $(C8_q)$ and $(C9_q)$ hold.
    Then, $\mu : \N \times \N^\N \to \N$ defined by 
    \begin{align} 
        \mu(k, f) = \zeta\Big(\wt{k}, \max\{\omega_3(12(\wt{k} + 1) - 1, \wt{f}), \rLimBetaOne(24K^2(\wt{k} + 1) - 1)\}\Big)
    \end{align}
    is a rate of metastability for $(x_n)$, where 
    \begin{align*}
        K &\geq \max\set{d(x_0, p), d(u, p)}, \\ 
        \wt{f}(i) &= 12K (\wt{k} + 1) (\ol{f}(i) + 1) (B(\ol{f}(i))) - 1, \\
        \ol{f}(i) &= f\left(\zeta\left(\wt{k}, \max\set{i, \rLimBetaOne(24K^2(\wt{k} + 1) - 1)}\right)\right), \\
        \wt{k} &= 4(k + 1)^2 - 1, \\ 
        \zeta(i, m) &= \rSumOneMinusBetaDivergent(m + \ceil{\ln(12K^2(i + 1))}) + 1, 
    \end{align*}
    and $\omega_3$ is defined as in Lemma \ref{prop:tmf:metastability:weak-seq-conv-elim},
    for $\Phi := \Psi$, with $\Psi$ given by Theorem \ref{thm:Tm-ar}.(i).
\end{theorem}
\begin{proof}
    Let $k \in \N$ and $f : \N \to \N$ be a monotone function. 
    Applying Proposition \ref{prop:tmf:metastability:weak-seq-conv-elim} 
    with $k := 12(\wt{k} + 1) - 1$ and $f := \wt{f}$, we get 
    an $N_0 \leq \omega_3(12(\wt{k} + 1) - 1, \wt{f})$ and an 
    $x \in \afixb{\wt{f}(N_0) + 1}{}$ such that
    \begin{align}
        \forall i \geq N_0 \left(\ql{x}{u}{x}{x_i} \leq \frac{1}{12(\wt{k} + 1)}\right).
    \end{align}

    We will apply Lemma ~\ref{lem:xu-metastable} with 
    \begin{align*}
        s_n &:= d^2(x_n, x), \\
        S &:= 4K^2,\\ 
        v_n &:= B(n) w_n = B(n) (2d(y_n, x)d(T_n x, x) + d^2(T_n x, x)), \\ 
        a_n &:= 1 - \beta_n, \\ 
        r_n &:= 2\beta_n \ql{x}{u}{x}{x_n} + (1 - \beta_n) d^2(x, u), \\ 
        q &:= f(\zeta(\wt{k}, n)), \\ 
        n &:= \max\set{N_0, \rLimBetaOne(24K^2(\wt{k} + 1) - 1)}, \\ 
        k &:= 4(k + 1)^2.  \\
    \end{align*}
    Let us first show that the hypotheses of the lemma are met. 
    Using \cite[Lemma~3.1]{Che23}, and the fact that $x \in \ballc{p}{K}$,
    it is clear that 
    \begin{align}
        d^2(x_i, x) \leq (d(x_i, p) + d(p, x))^2 \leq 4K^2
    \end{align}  
    holds for all $i \in \N$. Let now $i \in [n, q]$. 
    We want to show that $r_i \leq \dfrac{1}{3\wt{k} + 1}$ and that 
    $v_i \leq \dfrac{1}{(3\wt{k} + 1)(q + 1)}$.
    For $r_i$ we have:
    \begin{align*}
        r_i =&\ 2\beta_i \ql{x}{u}{x}{x_i} + (1 - \beta_i) d^2(x, u) \\ 
        \leq&\ 2\beta_i \frac{1}{12(\wt{k} + 1)} + (1 - \beta_i) d^2(x, u) \quad\text{as $i \geq N_0$} \\ 
        \leq&\ \frac{2}{12(\wt{k} + 1)} + (1 - \beta_i) d^2(x, u) \quad\text{as $\beta_i \leq 1$} \\ 
        \leq&\ \frac{2}{12(\wt{k} + 1)} + (1 - \beta_i) 4K^2 \\ 
        \leq&\ \frac{2}{12(\wt{k} + 1)} + \frac{4K^2}{24K^2(\wt{k} + 1)}  \quad\text{as $i \geq \rLimBetaOne(24K^2(\wt{k} + 1) - 1)$} \\ 
        \leq&\ \frac{2}{12(\wt{k} + 1)} + \frac{4K^2}{24K^2(\wt{k} + 1)} \\ 
        =&\ \frac{1}{6(\wt{k} + 1)} + \frac{1}{6(\wt{k} + 1)} 
        = \frac{1}{3(\wt{k} + 1)}. 
    \end{align*}

    Also, for $v_i$ we get:
    \begin{align*}
        v_i =&\ B(i) d(T_i x, x) (2d(y_i, x) + d(T_i x, x)) \\ 
        \leq&\ B(i) \frac{1}{\wt{f}(N_0) + 1} (2d(y_i, x) + d(T_i x, x)) \quad\text{as $x \in \afixb{\wt{f}(N_0) + 1}{}$} \\ 
        \leq&\ B(i) \frac{1}{\wt{f}(N_0) + 1} ((1 - \beta_i) d(u, x) + \beta_i d(x_i, x) + d(T_i x, x)) \quad\text{by (W1)} \\ 
        \leq&\ B(i) \frac{1}{\wt{f}(N_0) + 1} ((1 - \beta_i) 2K + \beta_i 2K + d(T_i x, x)) \\
        =&\ B(i) \frac{1}{\wt{f}(N_0) + 1} (2K + d(T_i x, x)) \\
        \leq&\ B(i) \frac{1}{\wt{f}(N_0) + 1} (2K + 2 d(p, x)) \quad\text{by nonexpansiveness and the fact that $p \in \fix{T_i}$}\\
        \leq&\ \frac{4 K B(i)}{\wt{f}(N_0) + 1} \\ 
        \leq&\ \frac{4K B(f(\zeta(\wt{k}, n)))} {\wt{f}(N_0) + 1} \quad\text{as $i \leq q = f(\zeta(\wt{k}, n))$ and $B$ is monotone} \\ 
        =&\ \frac{4K B(f(\zeta(\wt{k}, n)))}{12K(\wt{k} + 1)(\ol{f}(N_0) + 1)(B(\ol{f}(N_0)))} \\ 
        =&\ \frac{1}{3(\wt{k} + 1)(q + 1)} \quad\text{noting that, by definition, $\ol{f}(N_0) = f(\zeta(\wt{k}, n)) = q$}.
    \end{align*}
    Thus, we can apply Lemma \ref{lem:xu-metastable} to obtain that 
    \begin{align*}
        \forall i \in [\zeta(\wt{k}, n), q = f(\zeta(\wt{k}, n))] \left( d^2(x_i, x) \leq \frac{1}{\wt{k} + 1} \right),
    \end{align*}
    and therefore, taking into account that $\wt{k} = 4(k + 1)^2$,
    \begin{align*}
        \forall i \in [\zeta(\wt{k}, n), f(\zeta(\wt{k}, n))] \left( d(x_i, x) \leq \frac{1}{2(k + 1)} \right).
    \end{align*}
    Finally, to show the conclusion let $i, j \in [\zeta(\wt{k}, n), f(\zeta(\wt{k}, n))]$.
    Using the inequality above, we have 
    \begin{align*}
        d(x_i, x_j) \leq d(x_i, x) + d(x_j, x) \leq \frac{1}{2(k + 1)} + \frac{1}{2(k + 1)} \leq \frac{1}{k + 1}.
    \end{align*}
    
\end{proof}

    Switching from $\qcondSumOneMinusBetaDivergent$ to $\qcondProdBetaZero$
    and applying Lemma \ref{lem:xu-metastable}.(ii) instead of (i), we can also obtain the following result. 

    \begin{theorem}
        Let $X$ be a CAT(0) space, $(T_n : X \to X)$ be a family of nonexpansive self-mappings of $X$
        possessing common fixed points and let $p \in X$ be such a point.
        Suppose that $(T_n)$
        satisfies \eqref{eq:prePtwo} with respect to a sequence $(\gamma_n)$ of positive reals.
        Let $(x_n)$ be the sequence generated by \eqref{eq:tmf} and
        assume conditions $(C1_q^*)$, $(C2_q)$ -- $(C6_q)$, $(C8_q)$ and $(C9_q)$ hold.
        Then, $\mu : \N \times \N^\N \to \N$ defined by 
        \begin{align} 
            \mu(k, f) = \zeta^*\Big(\wt{k}, \max\{\omega_3(12(\wt{k} + 1) - 1, \wt{f}), \rLimBetaOne(24K^2(\wt{k} + 1) - 1)\}\Big)
        \end{align}
        is a rate of metastability for $(x_n)$, where 
        \begin{align*}
            K &\geq \max\set{d(x_0, p), d(u, p)}, \\ 
            \wt{f}(i) &= 12K (\wt{k} + 1) (\ol{f}(i) + 1) (B(\ol{f}(i))) - 1, \\
            \ol{f}(i) &= f\left(\zeta\left(\wt{k}, \max\set{i, \rLimBetaOne(24K^2(\wt{k} + 1) - 1)}\right)\right), \\
            \wt{k} &= 4(k + 1)^2 - 1, \\ 
            \zeta^*(i, m) &= \sigma^*(m, 12K^2(i + 1) - 1) + 1, 
        \end{align*}
        and $\omega_3$ is defined as in Lemma \ref{prop:tmf:metastability:weak-seq-conv-elim},
        for $\Phi := \Psi^*$, with $\Psi^*$ given by Theorem \ref{thm:Tm-ar}.(ii).
    \end{theorem}

\end{document}